\documentclass[a4paper, 12pt]{amsart}

\usepackage{amssymb}
\usepackage{amsmath}
\usepackage{amscd}
\usepackage{amsthm}
\usepackage[centertags]{amsmath}
\usepackage{amsfonts}
\usepackage{newlfont}
\usepackage{graphicx}
\usepackage[usenames]{color}
\usepackage{amsfonts, amssymb}
\usepackage{mathrsfs}
\usepackage{latexsym}
\usepackage{pstricks,pst-node,pst-fill}
\usepackage{tabulary,tabularx}
\usepackage[all]{xy}

\usepackage[latin1]{inputenc}

\newtheorem{theorem}{Theorem}[section]
\newtheorem{proposition}[theorem]{Proposition}
\newtheorem{corollary}[theorem]{Corollary}
\newtheorem{lemma}[theorem]{Lemma}

\theoremstyle{definition}
\newtheorem{remark}[theorem]{Remark}

\theoremstyle{definition}

\theoremstyle{remark}

\renewcommand{\u}{\mathfrak A}

\begin{document}
\title{An integral formula for multiple summing norms of operators}
\author{Daniel Carando \and Ver\'onica Dimant \and Santiago Muro \and Dami\'an Pinasco}
\thanks{This work was partially supported by CONICET PIP 0624, ANPCyT PICT 2011-1456, ANPCyT PICT 11-0738,
UBACyT 1-746 and UBACyT 20020130300052BA}

\address{Daniel Carando. Departamento de Matem\'{a}tica - Pab I,
Facultad de Cs. Exactas y Naturales, Universidad de Buenos Aires,
(1428) Buenos Aires, Argentina and IMAS-CONICET} \email{dcarando@dm.uba.ar}

\address{Ver\'onica Dimant. Departamento de Matem\'{a}tica, Universidad de San
Andr\'{e}s, Vito Dumas 284, (B1644BID) Victoria, Buenos Aires,
Argentina and CONICET} \email{vero@udesa.edu.ar}

\address{Santiago Muro. Departamento de Matem\'{a}tica - Pab I,
Facultad de Cs. Exactas y Naturales, Universidad de Buenos Aires,
(1428) Buenos Aires, Argentina and CONICET} \email{smuro@dm.uba.ar}

\address{Dami\'an Pinasco. Departamento de Matem\'aticas y Estad\'{\i}stica,
Universidad Torcuato di Tella, Av. F. Alcorta 7350, (1428), Ciudad Aut\'onoma de Buenos Aires, ARGENTINA and
CONICET}
\email{{\tt dpinasco@utdt.edu}}


\maketitle

\begin{abstract}
We prove that the multiple summing norm of multilinear operators defined on some $n$-dimensional real or
complex vector spaces with the $p$-norm may be
written as an integral with respect to stables measures. As an application we show inclusion and
coincidence results for multiple summing mappings. We also present some contraction properties  and compute or
estimate the limit orders of this class of operators.
\end{abstract}

\section*{Introduction}

The rotation invariance of the Gaussian measure on $\mathbb K^N$, which we will denote by $\mu_2^N$, allows us
to
show  the Khintchine equality. It asserts that if $c_{2,q}$ denotes the $q$-th moment of the one dimensional
Gaussian measure, and $\ell_2^N$ denotes $\mathbb K^N$ with the euclidean norm, then for
any $\alpha\in\mathbb K^N$, $1\le q<\infty$,
\begin{equation}\label{khintchine gaussiano}
c_{2,q}\|\alpha\|_{\ell_2^N}=
\Big(\int_{\mathbb K^N}|\langle\alpha,z\rangle|^qd\mu^N_2(z)\Big)^{1/q}.
\end{equation}
We may interpret this formula as follows: the norm of a linear functional $\alpha$ on $\ell_2^N$ is a multiple
of the $L^q$-norm of the linear functional with respect to the Gaussian measure on $\ell_2^N$. One may ask if
there is a formula like (\ref{khintchine gaussiano}) for linear functionals on some other space, or  even for
linear or multilinear operators. For linear functionals, an answer is provided by the $s$-stable L\'evy measure
(see for example \cite[24.4]{DefFlo93}): for $s< 2$ there exists a measure on $\mathbb K^N$, called the {\it
$s$-stable L\'evy measure} and denoted by $\mu_{s}$, which satisfies that for any $0<q<s$, $\alpha\in\mathbb
K^N$,
\begin{equation}\label{estable}
c_{s,q}\|\alpha\|_{\ell_s^N}=\Big(\int_{\mathbb K^N}|\langle\alpha,z\rangle|^qd\mu^N_{s}(z)\Big)^{1/q}, 
\end{equation}
where
$$c_{s,q}=\Big(\int_{\mathbb K}|z|^qd\mu_{s}^1(z)\Big)^{1/q}.$$
The question for linear operators is more subtle because there are many norms which are natural to consider on $\mathcal
L(\ell_2^N)$. The first result in this direction is due to Gordon
\cite{Gor69} (see also \cite[11.10]{DefFlo93}), who showed that the formula holds for the identity operator on
$\ell_2^N$, considering the absolutely $p$-summing norm of $id_{\ell_2^N}$, that is
$$\pi_p(id_{\ell_2^N})=c_{2,q}\Big(\int_{\mathbb K^N}\|z\|_{\ell_2^N}^q \, d\mu^N_2(z)\Big)^{1/q}.$$
Pietsch \cite{Pie72} extended this formula for arbitrary linear operators from $\ell_{s'}^N\to\ell_s^N$,
$s\ge2$ and used it to compute some limit orders (see also \cite[22.4.11]{Pie80}).

To generalize the formula to the multilinear setting there is again a new issue, because there are many
natural candidates of classes of multilinear operators that extend the ideal of absolutely $p$-summing linear
operators (for instance the articles \cite{PelSan11,Per05} are devoted to their comparison).
Among those candidates, the ideal of multiple summing multilinear operators is considered by many authors the
most important of these extensions and is also the most studied one. Some of the reasons are its
connections with the Bohnenblust-Hille inequality \cite{PerVil04}, or the results on the unconditional
structure of the space of multiple summing operators \cite{DefPer08}. Multiple summing operators were
introduced by Bombal, P\'erez-Garc\'ia and Villanueva \cite{BomPerVil04} and independently by Matos
\cite{Mat03}.
In this note we show that multiple summing operators constitute the correct framework for a multilinear
generalization of formula \eqref{khintchine gaussiano}. For this we present integral formulas for the exact
value of the multiple summing norm of multilinear forms and operators defined on $\ell_p^N$ for some values
of $p$. Moreover, we prove that for some other finite dimensional Banach spaces these formulas hold  up to
some constant independent of the dimension.

One particularity of the class of multiple summing operators on Banach spaces is that, unlike the linear
situation,
there is no general inclusion result. In \cite{BotMicPel10,Per04,Pop09} the authors investigate this problem
and prove several results showing that on some Banach spaces inclusion results hold, but on some other spaces
not. The integral formula for the multiple summing norm, together with Khintchine/Kahane type inequalities
will allow us to show some new coincidence and inclusion results for multiple summing operators.

Another application of these formulas deals with unconditionality in tensor products. Defant and
P\'erez-Garc\'ia showed in \cite{DefPer08} that the tensor norm associated to the ideal of multiple 1-summing
multilinear forms preserves unconditionality on $\mathcal L_r$ spaces. As a consequence of our formulas, we
give a simple proof of this fact for $\ell_r$ with $ r \ge 2$. Moreover, we show that
vector-valued multiple 1-summing
operators also satisfy a kind of unconditionality property in the appropriate range of Banach spaces.
Finally, we compute limit orders for the ideal of multiple summing operators.

Our main results are stated in Theorems~\ref{formulita escalar} and \ref{formulita}, which give an exact
formula for the multiple summing norm, and Proposition~\ref{normas equivalentes}, which gives integral
formulas for
estimating these norms in a wider range of spaces.

\section{Main results and their applications}

Let $E_1,\dots,E_m,F$ be real or complex Banach spaces. Recall that an $m$-linear operator
$T\in\mathcal L(E_1,\dots,E_m;F)$ is {\it multiple $p$-summing} if there exists $C>0$ such that for all finite sequences of vectors $(x^1_{j_1})_{j_1=1}^{J_1}\subset E_1,\dots,(x^m_{j_m})_{j_m=1}^{J_m}\subset E_m$
$$
\left(\sum_{j_1,\dots,j_m}\|T(x^1_{j_1},\dots,x^m_{j_m})\|_F^p\right)^{\frac1{p}}\le C
w_p((x_{j_1}^1)_{j_1})\dots w_p((x_{j_m}^m)_{j_m}),
$$
where $$w_p((y_j)_j)=\sup\left\{\Big(\sum_j|\gamma(y_j)|^p\Big)^{1/p}\,:\,\gamma\in B_{E'}\right\}.$$ The
infimum of all those constants $C$ is the multiple $p$-summing norm of $T$ and is denoted by $\pi_p(T)$. The
space of multiple $p$-summing multilinear operators is denoted by $\Pi_p(E_1,\dots,E_m;F)$. When
$E_1=\dots=E_m=E$, the spaces of continuous and multiple $p$-summing multilinear are denoted by $\mathcal
L(^mE;F)$ and $\Pi_p(^mE;F)$ respectively.

The following theorems are our main results. Their proofs will be given in Section~ \ref{sec-proofs}.

\begin{theorem}\label{formulita escalar}
Let $\phi$ be a multilinear form in $\mathcal L(^m\ell_r^N;\mathbb K)$, $p<r'<2$ or $r=2$. Then
$$
\pi_p(\phi)=\frac{1}{c_{r',p}^m}\  \Big(\int_{\mathbb K^N}\dots\int_{\mathbb
K^N}|\phi(z^{(1)},\dots,z^{(m)})|^pd\mu^N_{r'}(z^{(1)})\dots d\mu^N_{r'}(z^{(m)})\Big)^{1/p}.
$$
\end{theorem}

Before we state our second theorem, let us recall some necessary definitions and facts.
For $1\leq q \leq \infty$ and $1 \leq \lambda < \infty$ a normed space $X$ is called an
$\mathcal{L}_{q,\lambda}^g$\emph{-space}, if for each finite dimensional  subspace $M \subset X$ and
$\varepsilon >0$ there are $R \in \mathcal{L}(M,\ell_q^m)$ and $S \in \mathcal{L}(\ell_q^m, X)$ for some $m \in
\mathbb{N}$ factoring the inclusion map $I_M^X:M\to X$ such that $\|S\| \|R\| \leq \lambda + \varepsilon$:
\begin{equation} \label{facto}
\xymatrix{ M  \ar@{^{(}->}[rr]^{I_M^X} \ar[rd]^{R} & & {X} \\
& {\ell_q^m} \ar[ur]^{S} & }.
\end{equation}
$X$ is called an $\mathcal{L}_{q}^g$\emph{-space} if it is an $\mathcal{L}_{q,\lambda}^g$-space for some $\lambda \geq 1$.
Loosely speaking, $\mathcal{L}_{q}^g$-spaces share many properties of $\ell_q$, since they \emph{locally look like $\ell_q^m$}. The spaces $L_q(\mu)$ are $\mathcal{L}_{q,1}^g$-spaces.
For more information and properties of $\mathcal{L}_{q}^g$-spaces see \cite[Section 23]{DefFlo93}.

\begin{theorem}\label{formulita}
Let $T$ be a multilinear map in $\mathcal L(^m\ell_r^N;X)$, where $X$ is an $\mathcal L_{q,1}^g$-space and suppose $r$, $q$ and $p>0$ satisfy one of the following conditions
\begin{itemize}
\item[a)] $r=q=2$;
\item[b)] $r=2$ and either $p<q<2$ or $p=q$;
\item[c)] $p<r'<2$ and either $p<q\le 2$ or $p=q$.
\end{itemize}
Then
$$
\pi_p(T)= \frac{1}{c_{r',p}^m}\ \Big(\int_{\mathbb K^N}\dots\int_{\mathbb
K^N}\|T(z^{(1)},\dots,z^{(m)})\|_X^p \, d\mu^N_{r'}(z^{(1)})\dots d\mu^N_{r'}(z^{(m)})\Big)^{1/p}.
$$
\end{theorem}

It is clear that Theorem~\ref{formulita escalar} follows from Theorem~\ref{formulita}, but in fact, the proof of Theorem~\ref{formulita} uses the scalar result, which is much simpler and is interesting on its own.
We remark that the formula also holds for any  multilinear map in $\mathcal
L(\ell_{r_1}^N,\dots,\ell_{r_m}^N;X)$, where $X$ is an $\mathcal L_{q,1}^g$-space and
$r_1,\dots,r_m$, $q$ and $p$ satisfy conditions analogous to those of Theorem~\ref{formulita}. Moreover,  the
formula turns into an
equivalence between the $\pi_p$ norm and the integral if we take general $\mathcal L_q^g$-spaces.

On the other hand,  if we put $\ell_r$ in the domain, since multiple summing operators form a maximal ideal,
the formula holds with a limit over $N$ in the right hand side (here we consider $\mathbb K^N$ as a subset of
$\ell_r$).

There are situations  not covered by the previous theorem where  we have an equivalence or, at least, an
inequality between the $\pi_p$ and the $L_p(\mu_{s})$ norms.

\begin{proposition} \label{normas equivalentes} Let $T\in\mathcal L(^m\ell_r^N;X).$

($i$) Suppose either $r=2$ and $p,q<2;$ or $r=2$ and $q\le p$; or $p<r'<2$ and $q\le 2$. If $X$ is an $\mathcal L_{q}^g$-space, then we have
$$
\pi_p(T)\asymp \left(\int_{\mathbb K^N}\dots\int_{\mathbb
K^N}\|T(z^{(1)},\dots,z^{(m)})\|_X^p \, d\mu^N_{r'}(z^{(1)})\dots d\mu^N_{r'}(z^{(m)})\right)^{1/p},
$$
that is, the multiple $p$-summing and the $L_{r'}(\mathbb K^N\times\dots\times\mathbb
K^N,\mu_{r'}^N\times\dots\times\mu_{r'}^N)$ norm are equivalent in $\mathcal
L(^m\ell_r^N;X)$, with constants which are independent of $N$.

($ii$) If $r=2$ or $p<r'<2$ then we have, for any Banach space $X$,
$$
\pi_p(T)\succeq \left(\int_{\mathbb K^N}\dots\int_{\mathbb
K^N}\|T(z^{(1)},\dots,z^{(m)})\|_X^p \, d\mu^N_{r'}(z^{(1)})\dots d\mu^N_{r'}(z^{(m)})\right)^{1/p}.
$$
\end{proposition}

Now we describe some applications of these results.  The most direct one is an asymptotically correct
relationship between the multiple summing norm of a multilinear operator and the usual (supremum) norm.
Cobos, K\"uhn and Peetre \cite{CobKuhPee99} compared the Hilbert-Schmidt norm, $\pi_2$, with the usual norm of
multilinear forms. They showed that if $T$ is any $m$-linear form in $\mathcal L(^m\ell_2^N,\mathbb K)$ then
$$
\pi_2(T)\le N^{\frac{m-1}{2}}\|T\|.
$$
Moreover, the asymptotic bound is optimal in the sense that there exist constants $c_m$ and $m$-linear forms
$T$ on $\ell_2^N$ with $\|T\|=1$ and $\pi_2(T)\ge c_mN^{\frac{m-1}{2}}$. It is easy to see from this that the
correct exponent for the asymptotic bound for the Hilbert-valued case is $\frac{m}{2}$. The same holds for the
multiple $p$-summing norm for any $p$ because  all those norms are equivalent to the Hilbert-Schmidt norm in
$\mathcal L(^m\ell_2;\ell_2)$, see \cite{Mat03,Per04}. We see now that the same optimal exponent holds for multiple $p$-summing operators with values on $\mathcal L_q^g$-spaces.

First, note that passing to polar coordinates we have, in the complex case (the real case follows similarly)
\begin{eqnarray*}
& &\hspace{-5pt} \int_{\mathbb K^N}\dots\int_{\mathbb
K^N}\|T(z^{(1)},\dots,z^{(m)})\|_X^p \, d\mu^N_2(z^{(1)})\dots
d\mu^N_2(z^{(m)})  \\
& =&\hspace{-5pt} \frac{1}{\Gamma(N)^m}\int_{(
S^{2N-1})^m} \hspace{-5pt}\|T(\omega^{(1)},\dots,\omega^{(m)})\|_X^p \, d\sigma_{2N-1}(\omega^{(1)})\dots
d\sigma_{2N-1}(\omega^{(m)})\, \Big(\int_{0}^\infty 2\rho^{2N+p-1}e^{-\rho^2}d\rho\Big)^m\\
&\le &\hspace{-5pt} \|T\|^p\Big(\frac{\Gamma(N+p/2)}{\Gamma(N)}\Big)^{m},
\end{eqnarray*}
where $S^{2N-1}$ denotes the unit sphere in $\mathbb R^{2N}$ and $\sigma_{2N-1}$ the normalized Lebesgue
measure defined on it.

As a consequence of Proposition~\ref{normas equivalentes}, we obtain
\begin{equation}\label{desigualdad entre normas}
 \pi_p(T) \preceq \left(\frac{\Gamma(N+p/2)}{\Gamma(N)}\right)^{m/p}  \|T\| \preceq
N^{\frac{m}{2}}\|T\|
\end{equation}
for  $X$ a $\mathcal L_{q,\lambda}^g$-space and $p\ge q$ or $p,q<2$.

 Let us see that for $p,q\le 2$, the exponents are optimal. Since for any $T\in\mathcal
L(^m\ell_2^N;\ell_q)$ we have
$$
\Big(\sum_{j_1,\dots,j_m=1}^N\|T(e_{j_1},\dots,e_{j_m})\|_{\ell_q}
^p\Big) ^ { \frac1 {p}}\le \pi_p(T)N^{\frac{m}{p}-\frac{m}{2}}\preceq N^{\frac{m}{p}}\|T\|,
$$
it suffices to show that the inequality
\begin{equation}\label{triangular}
\Big(\sum_{j_1,\dots,j_m=1}^N\|T(e_{j_1},\dots,e_{j_m})\|_{\ell_q}
^p\Big) ^ { \frac1 {p}}\preceq N^{\frac{m}{p}}\|T\|
\end{equation}
 is optimal.
By \cite[Theorem 4]{Boa00}, there exist symmetric multilinear operators $\tilde T_N\in\mathcal
L(^m\ell_2^N,\ell_2^N)=\mathcal
L(^{m+1}\ell_2^N)$, such that,
$\displaystyle \tilde
T_N=\sum_{j_1,\dots,j_{m+1=1}}^N\varepsilon_{j_1,\dots,j_{m+1}}e_{j_1}\otimes\dots\otimes e_{j_{m+1}}$,
with $\varepsilon_{j_1,\dots,j_{m+1}}=\pm 1$ and $\|\tilde T_N\|\asymp \sqrt{N}$.

Let $T_N=i_{2q}\circ \tilde T_N$, where $i_{2q}:\ell_2^N\to\ell_q^N$ is the inclusion.
Then, $\displaystyle \|T_N\|\preceq N^\frac1{q}$ and
$$
\displaystyle\Big(\sum_{j_1,\dots,j_m=1}^N\|T_N(e_{j_1},\dots,e_{j_m})\|_{\ell_q}
^p\Big)^ { \frac1 {p}}N^{\frac1{q}+\frac{m}{p}}\succeq N^{\frac{m}{p}}\|T_N\|.
$$
This implies that inequality (\ref{triangular}) is optimal and, hence, so is \eqref{desigualdad entre normas}.

\subsection{Inclusion theorems}

The well-known inclusion theorem for absolutely summing linear operators states that for any Banach spaces
$E,F$ we have
$$
\Pi_s(E,F)\subset\Pi_t(E,F),\quad \textrm{ when }s\le t.
$$
Although multiple summing mappings share several properties of linear summing operators, there is no
general inclusion theorem in the multilinear case (see \cite{PerVil04}).
It is therefore interesting to investigate in which situations we do have inclusion type theorems.
The following theorem summarizes some of the most important known results on this topic.

\begin{theorem}[\cite{BotMicPel10,Per04,Pop09}]
($i$)  If $E$ has cotype $r\ge2$ then
$$
\Pi_s(^mE,F)=\Pi_1(^mE,F),\quad \textrm{ for }1\le s<r^*.
$$
($ii$) If $F$ has cotype $2$ then
$$
\Pi_s(^mE,F)\subset\Pi_2(^mE,F),\quad \textrm{ for }2\le s <\infty.
$$
\end{theorem}
The following picture illustrates the above theorem in the particular case where $E=\ell_2$ and $F=\ell_q$,
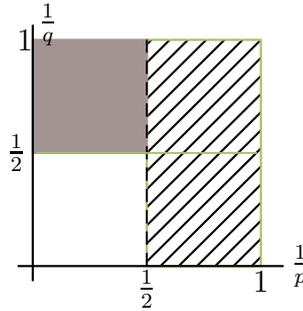
\[
\begin{pspicture}(3,3)
\pspolygon[linecolor=green!60!red!60,
fillstyle=hlines](1.5,0)(3,0)(3,3)(1.5,3)
\pspolygon[linecolor=black!60!pink!60,fillcolor=black!60!pink!60,
fillstyle=solid](1.5,3)(0,3)(0,1.5)(1.5,1.5)
\psline(0,-0.2)(0,3.2)\psline(-0.2,0)(3.3,0)
\rput[l](3.4,0){$\frac1{p}$}
\rput[d](0.2,3.2){$\frac1{q}$}
\rput[u](3,-0.2){$1$}
\rput[r](0,3){$1$}
\rput[u](1.5,-0.3){$\frac1{2}$}
\rput[r](-0.1,1.5){$\frac1{2}$}
\psline[linestyle=dashed,linewidth=0.7pt](1.5,0)(1.5,3)
\psline[linewidth=0.7pt,linecolor=green!60!red!60](0.05,1.5)(2.95,1.5)
\end{pspicture}
\]
$\smallskip$

In the ruled area we have $\Pi_{p_1}(^m\ell_2;\ell_q)=\Pi_{p_2}(^m\ell_2;\ell_q)$ and in the shaded area
we have the reverse inclusion $\Pi_{p_1}(^m\ell_2;\ell_q)\subset\Pi_{2}(^m\ell_2;\ell_q)$ for
$p_1\ge 2$.

As a consequence of our integral formula, we obtain the following improvement to the previous result, which will be proved in Section \ref{sec-proofs}.

\begin{proposition}\label{prop inclusion} Let $Y$ be a $\mathcal L_2^g$-space and $X$ a $\mathcal L_q^g$-space.

If $p\ge q$, then $\Pi_p(^mY;X)=\Pi_q(^mY;X)$.

If $p\le q$, then $\Pi_p(^mY;X)\subset \Pi_q(^mY;X)$.
\end{proposition}

With the information given by the above proposition, we have the following new picture.
\medskip
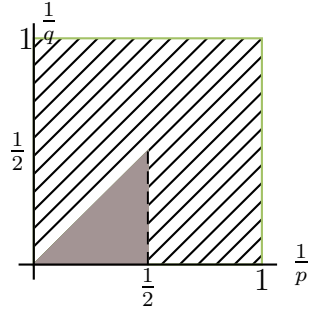
\[
\begin{pspicture}(3,3)
\pspolygon[linecolor=green!60!red!60,
fillstyle=hlines](1.5,0)(3,0)(3,3)(0,3)(0,0)(1.5,1.5)
\pspolygon[linecolor=black!60!pink!60,fillcolor=black!60!pink!60,
fillstyle=solid](1.5,1.5)(0,0)(1.5,0)
\psline(0,-0.2)(0,3.2)\psline(-0.2,0)(3.2,0)
\rput[l](3.4,0){$\frac1{p}$}
\rput[d](0.2,3.2){$\frac1{q}$}
\rput[u](3,-0.2){$1$}
\rput[r](0,3){$1$}
\rput[u](1.5,-0.3){$\frac1{2}$}
\rput[r](-0.1,1.5){$\frac1{2}$}
\psline[linestyle=dashed,linewidth=0.7pt](1.5,0)(1.5,1.5)
\end{pspicture}
\]
$\smallskip$
\textsl{}
In the ruled area we have $\Pi_{p_1}(^m\ell_2;\ell_q)=\Pi_{p_2}(^m\ell_2;\ell_q)$ and in the shaded
area
we have the (direct) inclusion $\Pi_{p}(^m\ell_2;\ell_q)\subset\Pi_{q}(^m\ell_2;\ell_q)$ for
$p\le q$.

\subsection{A contraction result and unconditionality}

Let us begin with this contraction result for the $p$-summing norm of multilinear operators.

\begin{theorem}\label{contraction}
Suppose $X$ is a $\mathcal L_q^g$-space and let $r$, $q$ and $p>0$ satisfy  one of the conditions in
Proposition~\ref{normas equivalentes} ($i$). Then, there is a constant $K$ (depending on $r$, $q$ and $p$),
such that for any finite matrix $(x_{i_1,\dots,i_m})_{i_1,\dots,i_m}\subset X$ and any choice of scalars
$\alpha_{i_1,\dots,i_m}$ we have,
$$
\pi_p\left( \sum _{i_1,\dots,i_m} \alpha_{i_1,\dots,i_m}\  e_{i_1}'\otimes\cdots\otimes e_{i_m}' \
x_{i_1,\dots,i_m} \right) \le K \|(\alpha_{i_1,\dots,i_m})\|_\infty\ \pi_p\left( \sum _{i_1,\dots,i_m}
e_{i_1}' \otimes\cdots\otimes e_{i_m}' \ x_{i_1,\dots,i_m} \right),
$$
where the $\pi_p$ norms are taken in $\Pi_p(^m\ell_r;X)$.
\end{theorem}
\begin{proof} If we show the inequality  for $\alpha_{i_1,\dots,i_m}=\pm 1$, standard procedures lead to the
desired inequality for general scalars, eventually with different constants (see, for example, Section 1.6 in
\cite{DieJarTon95}).
We set
 $$T=\sum _{i_1,\dots,i_m}  e_{i_1}' \otimes\cdots\otimes e_{i_m}' \ x_{i_1,\dots,i_m} \quad \text{and}\quad T_\alpha=\sum _{i_1,\dots,i_m} \alpha_{i_1,\dots,i_m} e_{i_1}' \otimes\cdots\otimes e_{i_m}' \ x_{i_1,\dots,i_m}$$
and let $(r_k)_k$ be the sequence of Rademacher functions on $[0,1]$. For any choice of  $t_1\dots,t_m\in [0,1]$, we have
\begin{eqnarray*}
& &
\int_{\mathbb K^N}\dots\int_{\mathbb K^N}\|T_\alpha(r_1(t_1)z^{(1)},\dots,r_n(t_m)z^{(m)})\|_X^p
d\mu^N_{r'}(z^{(1)})\dots d\mu^N_{r'}(z^{(m)})  \\
&=& \int_{\mathbb K^N}\dots\int_{\mathbb K^N}\|T_\alpha(z^{(1)},\dots,z^{(m)})\|_X^p d\mu^N_{r'}(z^{(1)})\dots
d\mu^N_{r'}(z^{(m)}).\end{eqnarray*}
We integrate on $t_j\in [0,1]$, $j=1,\dots,m$ and use Fubini's theorem to obtain
\begin{eqnarray}
& & \int_{\mathbb K^N}\dots\int_{\mathbb K^N}\|T_\alpha(z^{(1)},\dots,z^{(m)})\|_X^p d\mu^N_{r'}(z^{(1)})\dots
d\mu^N_{r'}(z^{(m)}) \label{conalfa}\\
&=& \hspace{-8pt} \int_{\mathbb K^N}\dots\int_{\mathbb K^N} \int_0^1 \dots \int_0^1
\big\|T_\alpha(r_1(t_1)z^{(1)},\dots,r_n(t_m)z^{(m)})\big\|_X^p dt_1\dots dt_m d\mu^N_{r'}(z^{(1)})\dots
d\mu^N_{r'}(z^{(m)})\nonumber \\
&=& \hspace{-8pt}
\int_{[0;1]^m}\int_{(\mathbb {K^N})^m} \hspace{-3pt}
\big\|\hspace{-8pt} \sum_{i_1,\dots,i_m}\hspace{-6pt} r_{i_1}(t)\dots r_{i_m}(t) \alpha_{i_1,\dots,i_m} z^{(1)}_{i_1} \cdots
z^{(m)}_{i_m} x_{i_1,\dots,i_m} \big\|_X^p \hspace{-1pt} dt_1\ldots dt_m d\mu^N_{r'}(z^{(1)}\hspace{-1pt})\dots
d\mu^N_{r'}(z^{(m)}\hspace{-1pt}).\nonumber
\end{eqnarray}
Since $X$ has nontrivial cotype  and local unconditional structure, we can apply a multilinear version of Pisier's deep result \cite[Proposition 2.1]{Pis78} (which follows the same lines as the bilinear result) to show that, for any $z^{(1)},\dots,z^{(m)} \in \mathbb K^N$, we have
\begin{eqnarray*}
& &
\int_0^1 \dots \int_0^1
\big\| \sum _{i_1,\dots,i_m} r_{i_1}(t)\dots r_{i_m}(t)\ \alpha_{i_1,\dots,i_m}\  z^{(1)}_{i_1} \cdots  z^{(m)}_{i_m} \ x_{i_1,\dots,i_m}  \big\|_X^2 dt_1\dots dt_m \\
&\le & K_X \ \int_0^1 \dots \int_0^1
\big\| \sum _{i_1,\dots,i_m} r_{i_1}(t)\dots r_{i_m}(t)\ z^{(1)}_{i_1} \cdots  z^{(m)}_{i_m} \ x_{i_1,\dots,i_m}  \big\|_X^2 dt_1\dots dt_m
\end{eqnarray*}
Using a multilinear Kahane inequality (which may be proved by induction on $m$), the same holds, with
a different constant, if we consider the power $p$ in the integrals.
This means that we can take the $\alpha_{i_1,\dots,i_k}$ from \eqref{conalfa}, paying the price of a constant $K$. Now, we can go all the way back as before to obtain
\begin{eqnarray*} & &
\int_{\mathbb K^N}\dots\int_{\mathbb K^N}\|T_\alpha(z^{(1)},\dots,z^{(m)})\|_X^p d\mu^N_{r'}(z^{(1)})\dots
d\mu^N_{r'}(z^{(m)})\\ &\le & K \int_{\mathbb K^N}\dots\int_{\mathbb K^N}\|T(z^{(1)},\dots,z^{(m)})\|_X^p
d\mu^N_{r'}(z^{(1)})\dots d\mu^N_{r'}(z^{(m)}).
\end{eqnarray*}
The integral formula in Proposition~\ref{normas equivalentes} gives the result.
\end{proof}

Note that, in the scalar valued case, the previous theorem asserts that the monomials form an unconditional
basic sequence in $\Pi_1(^m\ell_r)$ for $r\ge 2$. This is a particular case of the result of Defant and
Pérez-García in \cite{DefPer08}. It should be noted that the analogous scalar valued result is much easier to
prove: after introducing the Rademacher functions as in the previous proof, we just have to use a multilinear
Khintchine inequality and the integral formula from Theorem~\ref{formulita escalar} to obtain the result
(Pisier's result is, of course, not needed in this case).

\subsection{Limit orders}

As a consequence of the integral formula for the $p$-summing norm, we are able to compute limit orders of
multiple summing operators (see definitions below). Limit orders of the ideal of scalar valued multiple 1-summing forms were computed
in \cite{DefPer08} for the bilinear case. In the multlinear case, they were computed in  \cite{CarDimSev07}
for  $\ell_r$ with $1\le r\le 2$ and in  \cite{Paltesis} for  $\ell_r$ with $r\ge 2$. This latter case can be
easily obtained from our integral formula for the multiple summing norm.
We will actually use the integral formula to compute some limit orders for the vector
valued case, the mentioned scalar case being very similar. 

A subclass $\mathfrak{A}$ of the class $\mathcal{L}$
of all $m$-linear continuous mappings between Banach spaces is called an {\it ideal
of $m$-linear mappings} if
\begin{enumerate}
 \item For all Banach spaces $E_1,\dots,E_m,F$, the component set $\mathfrak{A}(E_1,\dots,E_m;F):=\mathfrak{A}\cap\mathcal(E_1,\dots,E_m;F)$ is a linear subspace of $\mathcal(E_1,\dots,E_m;F)$.

 \item If $T_j\in\mathcal(E_j;G_j)$,  $\phi\in\mathfrak{A}(G_1,\dots,G_m;G)$ and $S\in\mathcal L(G,F)$, then $S\circ\phi\circ(T_1,\dots,T_m)$ belongs to $\mathfrak{A}(E_1,\dots,E_m;F)$.

 \item The application $\mathbb K^m\ni(\lambda_1,\dots,\lambda_m)\mapsto \lambda_1\cdot\ldots\cdot\lambda_m\in\mathbb K$ is in $\mathfrak{A}(\mathbb K,\dots,\mathbb K;\mathbb K)$.
\end{enumerate}
A {\it normed ideal of $m$-linear  operators}
$(\u,\|\cdot\|_{\u})$ is an ideal $\u$ of $m$-linear operators together with an ideal norm $\|\cdot\|_{\u}$, that is,
\begin{enumerate}
 \item  $\|\cdot\|_{\u}$ restricted to each component is a norm.

 \item If $T_j\in\mathcal(E_j;G_j)$,  $\phi\in\mathfrak{A}(G_1,\dots,G_m;G)$ and $S\in\mathcal L(G,F)$, then $\|S\circ\phi\circ(T_1,\dots,T_m)\|_{\u}\le\|S\|\|\phi\|_{\u}\|T_1\|\cdot\dots\cdot\|T_m\|$.

 \item $\|\mathbb K^m\ni(\lambda_1,\dots,\lambda_m)\mapsto \lambda_1\cdot\ldots\cdot\lambda_m\in\mathbb K\|_{\mathfrak{A}}=1$.
\end{enumerate}

Given a normed ideal of $m$-linear  operators
$(\u,\|\cdot\|_{\u})$, the {\it limit order}
$\lambda_m(\u,r,q)$ is defined as the infimum of all  $\lambda\ge0$ such that there is a constant $C>0$
satisfying
$$
\|\Phi_N\|_{\u}\le CN^\lambda,
$$
for every $N\ge1$, where $\Phi_N:\ell_r^N\times\dots\times\ell_r^N\to\ell_q^N$  is the $m$-linear operator,
$\Phi_N(x^1,\dots,x^m)=\sum_{j=1}^Nx^1_j\dots x^m_je_j$.
\begin{proposition}
$$
\lambda_m(\Pi_1,r,q)=\left\{\begin{array}{lll}
                      \frac1{q} &\textrm{ if }&q\le r'\le 2\\
		      \frac1{r'}&\textrm{ if }&r'\le q\le 2\\
		      \frac1{q}+\frac{m}{2}-\frac{m}{r}&\textrm{ if }&  \frac{2mq}{2+mq}< r\le 2\textrm{
and }q\le 2\\
		      0&\textrm{ if }&1\le r\le \frac{2mq}{2+mq}
                     \end{array}\right.
$$
\end{proposition}
These values can be represented by the following picture:
$$
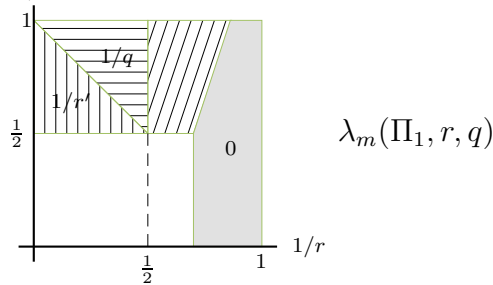
\begin{pspicture}(3,3)
\pspolygon[linecolor=green!60!red!60,linewidth=0.2pt,
fillstyle=hlines,hatchangle=71,hatchwidth=0.1pt](1.5,1.5)(2.1,1.5)(2.6,3)(1.5,3)
\pspolygon[linecolor=green!60!red!60,linewidth=0.2pt,fillcolor=gray!40!white!60,
fillstyle=solid	](2.1,0)(3,0)(3,3)(2.585,3)(2.1,1.5)
\pspolygon[linecolor=green!60!red!60,linewidth=0.2pt,
fillstyle=hlines,hatchangle=90,hatchwidth=0.1pt](1.5,1.5)(0,3)(0,1.5)
\pspolygon[linecolor=green!60!red!60,linewidth=0.2pt,
fillstyle=hlines,hatchangle=0,hatchwidth=0.1pt](1.5,1.5)(0,3)(1.5,3)
\psline(0,-0.2)(0,3.2)\psline(-0.2,0)(3.2,0)
\rput[d](0,3.4){$_{1/{q}}$}
\rput[l](3.4,0){$_{1/{r}}$}
\rput[u](3,-0.2){$_1$}
\rput[r](0,3){$_1$}
\rput[u](1.5,-0.3){$_\frac1{2}$}
\rput[r](-0.1,1.5){$_\frac1{2}$}
\psline[linestyle=dashed,linewidth=0.2pt](1.5,0)(1.5,1.5)
\rput[l](4,1.5){$\lambda_m(\Pi_1,r,q)$}
\rput[c](1.1,2.5){$_{1/q}$}
\rput[c](2.6,1.3){$_0$}
\rput[c](.5,1.9){$_{1/r'}$}
\end{pspicture}
$$

\smallskip
The proof will be splitted in several lemmas.
\begin{lemma}
Let $p\le q\le r'\le 2$ then $\lambda_m(\Pi_p,r,q)=\frac1{q}$.
\end{lemma}
\begin{proof}
Let $p\le q< r'\le 2$. Then by Theorem \ref{formulita},
\begin{eqnarray*}
c_{r',p}^m\pi_p(\Phi_N)
&=& \Big(\int_{\mathbb K^N}\dots\int_{\mathbb K^N}\Big(\sum_j|z^{(1)}_j\dots
z^{(m)}_j|^q\Big)^{p/q}d\mu^N_{r'}(z^{(1)})\dots d\mu^N_{r'}(z^{(m)})\Big)^{1/p}\\
&\le & \Big(\int_{\mathbb K^N}\dots\int_{\mathbb K^N}\sum_j|z^{(1)}_j\dots
z^{(m)}_j|^qd\mu^N_{r'}(z^{(1)})\dots d\mu^N_{r'}(z^{(m)})\Big)^{1/q}\\
&= & c_{r',q}^{m}N^{1/q}. \\
\end{eqnarray*}
Thus, $\lambda_m(\Pi_p,r,q)\le \frac1{q}$. On the other hand,
\begin{eqnarray*}
c_{r',p}^m\pi_p(\Phi_N) &=&  \Big(\int_{\mathbb K^N}\dots\int_{\mathbb K^N}\Big(\sum_j|z^{(1)}_j\dots
z^{(m)}_j|^q\Big)^{p/q}d\mu^N_{r'}(z^{(1)})\dots d\mu^N_{r'}(z^{(m)})\Big)^{1/p}\\
&\ge & \Big(\int_{\mathbb K^N}\dots\int_{\mathbb K^N}N^{p/q-1}\sum_j|z^{(1)}_j\dots
z^{(m)}_j|^pd\mu^N_{r'}(z^{(1)})\dots d\mu^N_{r'}(z^{(m)})\Big)^{1/p}\\
&= & c_{r',p}^{m}N^{1/q}.
\end{eqnarray*}
Hence, $\lambda_m(\Pi_p, r, q)\ge \frac{1}{q}$ and the proof is done.
\end{proof}

\begin{lemma}
Let $p<r'\le q<2$. Then $\lambda_m(\Pi_p,r,q)= \frac1{r'}$.
\end{lemma}
\begin{proof}
Let $1<s<r'\le q<2$. Then, by Theorem \ref{formulita},
  \begin{eqnarray*}
c_{r',1}^m\pi_1(\Phi_N)
&=& \int_{\mathbb K^N}\dots\int_{\mathbb K^N}\Big(\sum_j|z^{(1)}_j\dots
z^{(m)}_j|^q\Big)^{1/q}d\mu^N_{r'}(z^{(1)})\dots d\mu^N_{r'}(z^{(m)})\\
&\le& \int_{\mathbb K^N}\dots\int_{\mathbb K^N}\Big(\sum_j|z^{(1)}_j\dots
z^{(m)}_j|^s\Big)^{1/s}d\mu^N_{r'}(z^{(1)})\dots d\mu^N_{r'}(z^{(m)})\\
&\le & \Big(\int_{\mathbb K^N}\dots\int_{\mathbb K^N}\sum_j|z^{(1)}_j\dots
z^{(m)}_j|^sd\mu^N_{r'}(z^{(1)})\dots d\mu^N_{r'}(z^{(m)})\Big)^{1/s}\\
&= & c_{r',s}^{m}N^{1/s}. \\
\end{eqnarray*}
Since this is true for every $s<r'$, $\lambda_m(\Pi_1,r,q)\le \frac1{r'}$.

On the other hand, let $\Psi_N:\ell_r^N\times\dots\ell_r^N\times\ell_{q'}^N\to \mathbb C$, the $(m+1)$-linear
form  induced by $\Phi_N$. By \cite[Proposition 2.2]{PerVil04} or
\cite[Proposition 2.5]{Mat03}, $\pi_1(\Phi_N)\ge\pi_1(\Psi_N)$. Thus, by Theorem \ref{formulita
escalar} taking into account the comments after Theorem \ref{formulita}, we have
\begin{eqnarray*}
c_{r',1}^mc_{q,1}\pi_1(\Psi_N) &=& \int_{\mathbb K^N}\dots\int_{\mathbb
K^N}|\Psi_N(z^{(1)},\dots,z^{(m+1)})|d\mu^N_{r'}(z^{(1)})\dots d\mu^N_{r'}(z^{(m)})d\mu^N_{q}(z^{(m+1)})\\
&=& \int_{\mathbb K^N}\dots\int_{\mathbb K^N}\Big|\sum_jz^{(1)}_j\dots
z^{(m+1)}_j\Big|d\mu^N_{r'}(z^{(1)})\dots
d\mu^N_{r'}(z^{(m)})d\mu^N_{q}(z^{(m+1)})\\
&=& c_{r',1} \int_{\mathbb K^N}\dots\int_{\mathbb K^N}\Big(\sum_j|z^{(2)}_j\dots
z^{(m+1)}_j|^{r'}\Big)^{1/r'}d\mu^N_{r'}(z^{(2)})\dots d\mu^N_{r'}(z^{(m)})d\mu^N_{q}(z^{(m+1)})\\
&\ge& c_{r',1}N^{\frac1{r'}-1} \int_{\mathbb K^N}\dots\int_{\mathbb K^N}\sum_j|z^{(2)}_j\dots
z^{(m+1)}_j|d\mu^N_{r'}(z^{(2)})\dots d\mu^N_{r'}(z^{(m)})d\mu^N_{q}(z^{(m+1)})\\
&=& c_{r',1}^mc_{q,1}N^{\frac1{r'}-1}N \,=\, c_{r',1}^mc_{q,1}N^{\frac1{r'}}
\end{eqnarray*}
Therefore $\lambda_m(\Pi_1,r,q)= \frac1{r'}$.

This proves our assertions for $p=1$. By \cite[Theorem 4.7]{BotMicPel10},
$\Pi_p(\ell_r,\ell_q)=\Pi_1(\ell_r,\ell_q)$ for every $1\le p\le 2$, and the lemma follows.
\end{proof}

Since $\ell_r$ has cotype 2 for $1\le r\le 2$,  given  any $m$-linear form $T\in \mathcal
L(\ell_r^N,\dots,\ell_r^N; \mathbb C)$, we know from
\cite[Lemma 4.5]{DefPer08} that
\begin{equation}\label{lemadeDP}
\pi_1(T)\asymp \sup \pi_1(T\circ (D_{\sigma_1},\dots,D_{\sigma_m})),
\end{equation}
where the supremum is taken over the set of norm one diagonal operators $D_{\sigma_j}:\ell_2^N\to \ell_r^N$. The vector-valued version of this result follows the same lines, so \eqref{lemadeDP} holds for any $m$-linear map $T\in \mathcal
L(\ell_r^N,\dots,\ell_r^N; Y)$, for every Banach space $Y$.

\begin{lemma}
Let $1\le p,q,r\le2$. Then

$(i)$ $\lambda_m(\Pi_p,r,q)=0$ for $1\le r\le \frac{2mq}{2+mq}$.

$(ii)$ $\lambda_m(\Pi_p,r,q)=\frac1{q}+\frac{m}{2}-\frac{m}{r}$ for $\frac{2mq}{2+mq}< r\le 2$.
\end{lemma}
\begin{proof}
Let $\frac1{t}=\frac1{r}-\frac12$, then for any diagonal operator we have $\|D_\sigma\|_{\mathcal L(\ell_2^N;\ell_r^N)}=\|\sigma\|_{\ell_t^N}$.
Since $\Phi_N\circ (D_{\sigma_1},\dots,D_{\sigma_m})\in\mathcal L(^m\ell_2^N;\ell_q^N)$, by Theorem
\ref{formulita} we have
\begin{equation*}\label{pi_1 con diagonales}
\pi_1(\Phi_N\circ (D_{\sigma_1},\dots,D_{\sigma_m}))\asymp
\int_{\mathbb K^N}\dots\int_{\mathbb
K^N}\Big(\sum_{j=1}^N |\sigma_1(j)z_j^{(1)}\dots\sigma_m(j)z_j^{(m)}|^q\Big)^{1/q}d\mu^N_{2}(z^{(1)})\dots
d\mu^N_{2}(z^{(m)}).
\end{equation*}

$(i)$ The assumption $1\le r\le \frac{2mq}{2+mq}$ implies $t\le mq$. Then
$$
\Big(\sum_{j=1}^N |\sigma_1(j)z_j^{(1)}\dots\sigma_m(j)z_j^{(m)})|^q\Big)^{1/q}\le \|\sigma_1\|_{\ell_t^N}\dots\|\sigma_m\|_{\ell_t^N}\sup_{j} |z_j^{(1)}\dots z_j^{(m)}|.
$$
Consequently, for any $s\ge1$, we have
\begin{align*}
\pi_1(\Phi_N\circ (D_{\sigma_1},\dots,D_{\sigma_m})) &\preceq \int_{\mathbb K^N}\dots\int_{\mathbb
K^N}\sup_{j} |z_j^{(1)}\dots z_j^{(m)}|d\mu^N_{2}(z^{(1)})\dots
d\mu^N_{2}(z^{(m)}) \\
&\le \Big(\int_{\mathbb K^N}\dots\int_{\mathbb
K^N}\sum_{j=1}^N |z_j^{(1)}\dots z_j^{(m)}|^s d\mu^N_{2}(z^{(1)})\dots
d\mu^N_{2}(z^{(m)})\Big)^{1/s}   =  c_{2,s}^mN^\frac1{s},
\end{align*}
which implies that $\lambda_m(\Pi_1,r,q)=0$.

$(ii)$ The assumption $\frac{2mq}{2+mq}\le r< 2$ implies $t>mq$. Let $\frac1{q}=\frac{m}{t}+\frac1{s}$. Then
$$
\Big(\sum_{j=1}^N |\sigma_1(j)z_j^{(1)}\dots\sigma_m(j)z_j^{(m)}|^q\Big)^{1/q}\le \|\sigma_1\|_{\ell_t^N}\dots\|\sigma_m\|_{\ell_t^N}\Big(\sum_{j=1}^N |z_j^{(1)}\dots z_j^{(m)})|^s\Big)^{1/s}.
$$
Thus we have,
\begin{eqnarray*}
\pi_1(\Phi_N\circ (D_{\sigma_1},\dots,D_{\sigma_m})) &\preceq& \int_{\mathbb K^N}\dots\int_{\mathbb
K^N}\Big(\sum_{j=1}^N |z_j^{(1)}\dots z_j^{(m)}|^s\Big)^{1/s}d\mu^N_{2}(z^{(1)})\dots
d\mu^N_{2}(z^{(m)})\\
&\le & \Big(\int_{\mathbb K^N}\dots\int_{\mathbb
K^N}\sum_{j=1}^N |z_j^{(1)}\dots z_j^{(m)}|^s d\mu^N_{2}(z^{(1)})\dots
d\mu^N_{2}(z^{(m)})\Big)^{1/s}\\
&=& c_{2,s}^{m} N^{1/s}\quad \asymp\, N^{\frac1{q}+\frac{m}{2}-\frac{m}{r}}.
\end{eqnarray*}
On the other hand,
\begin{align*}
\pi_1(\Phi_N\circ (D_{\sigma_1},\dots,D_{\sigma_m})) &\succeq
 N^{-1/q'}\int_{\mathbb K^N}\dots\int_{\mathbb K^N}
\sum_{j=1}^N |\sigma_1(j)z_j^{(1)}\dots\sigma_m(j)z_j^{(m)}|d\mu^N_{2}(z^{(1)})\dots
d\mu^N_{2}(z^{(m)})\\
&= N^{-1/q'}\,c_{2,1}^m\,
\sum_{j=1}^N |\sigma_1(j)\dots\sigma_m(j)|.
\end{align*}
Taking supremum over $\sigma_k\in B_{\ell_t^N}$, $k=1,\dots,m$, and using \eqref{lemadeDP} we get that $$\pi_1(\Phi_N) \succeq  N^{-1/q'} N^{1-m/t} c_{2,1}^m  \asymp\, N^{\frac1{q}+\frac{m}{2}-\frac{m}{r}}.$$
This proves our assertions for $p=1$. Since $\ell_r$ has cotype 2, by \cite[Theorem 4.6]{BotMicPel10},
$\Pi_p(\ell_r,\ell_q)$ coincides with $\Pi_1(\ell_r,\ell_q)$ for every $1\le p\le 2$, and the lemma follows.
\end{proof}

\section{Proofs of the main results}\label{sec-proofs}

The proofs will be splitted in a few lemmas. We will also use the following result, which is \cite[Proposition 3.1]{Per04}.
\begin{proposition}[P\'erez-Garc\'ia]\label{perez}
Let $T\in\Pi_p^m(X_1,\dots,X_m;Y)$ and let $(\Omega_j,\mu_j)$ be measure spaces for each $1\le j\le m$. We
have
\begin{multline*}
\Big(\int_{\Omega_1}\dots\int_{\Omega_m}\|T(f_1(w_1),\dots,f_m(w_m))\|_Y^p \, d\mu_1(w_1)\dots
d\mu_m(w_m)\Big)^{1/p} \\ \le \pi_p(T) \prod_{j=1}^m\sup_{x_j^*\in B_{X_j^*}} \Big(\int_{\Omega_j} |\langle
x_j^*,f_j(w_j)\rangle|^pd\mu_j(w_j)\Big)^{1/p},
\end{multline*}
for every $f_j\in L_p (\mu_j,X_j)$.
\end{proposition}

A simple consequence of this proposition is the following.

\begin{lemma}\label{formulita desig facil}
Let $T$ be a multilinear operator in $\mathcal L(^m\ell_r^N;Y)$, and $p<r'<2$ or $r=2$.
Then
$$
\Big(\int_{\mathbb K^N}\dots\int_{\mathbb K^N}\|T(z^{(1)},\dots,z^{(m)})\|_Y^p \, d\mu^N_{r'}(z^{(1)})\dots
d\mu^N_{r'}(z^{(m)})\Big)^{1/p}\le c_{r',p}^m\pi_p(T).
$$
\end{lemma}
\begin{proof}
Let $(\Omega_j,\mu_j)=(\mathbb K^N,\mu_{r'})$, $f_j\in L_p((\mathbb K^N,\mu_{r'}),\mathbb K^N)$, $f_j(z)=z$
for all $j$ and $p<r'<2$ or $r=2$. By Proposition \ref{perez} and rotation invariance of stable measures,
\begin{multline*}
\Big(\int_{\mathbb K^N}\dots\int_{\mathbb K^N}\|T(z^{(1)},\dots,z^{(m)})\|_Y^p \, d\mu^N_{r'}(z^{(1)})\dots
d\mu^N_{r'}(z^{(m)})\Big)^{1/p} \\
\le \pi_p(T) \prod_{j=1}^m\sup_{w_j\in B_{\ell_{r'}^N}} \Big(\int_{\mathbb K^N} |\langle
z^{(j)},w_j\rangle|^pd\mu^N_{r'}(z^{(j)})\Big)^{1/p}\\
= \pi_p(T)\Big(\int_{\mathbb K^N} |e_1'(z)|^pd\mu^N_{r'}(z)\Big)^{m/p} \; = \; \pi_p(T)c_{r',p}^{m}. \qedhere
\end{multline*}
\end{proof}

Now we are ready to prove Theorem~\ref{formulita escalar}.

\begin{proof} \emph{of Theorem~\ref{formulita escalar}}
One inequality is given in the previous Lemma.
We prove the reverse inequality by induction on $m$. For $m=1$, we have $\phi\in(\ell_r^N)'=\ell_{r'}^N$ and then
\begin{eqnarray*}
\pi_p(\phi) & = &  \|\phi\|_{\ell_{r'}^N} \; = \;\Big(\sum_{j=1}^N |e_j'(\phi)|^{r'}\Big)^{1/r'} = \;
c_{r',p}^{-1}\Big(\int_{\mathbb K^N}\Big|\sum_{j=1}^N e_j'(\phi) z_j\Big|^pd\mu^N_{r'}(z)\Big)^{1/p}  \\
  & = & c_{r',p}^{-1}\Big(\int_{\mathbb K^N}|\phi(z)|^pd\mu^N_{r'}(z)\Big)^{1/p} .
 \end{eqnarray*}
Suppose that for any $k$-linear  form $\psi:\ell_r^N\times\dots\times\ell_r^N\to\mathbb K$, with $k<m$,
we have,
\begin{eqnarray*}
\sum_{n_1,\dots,n_k}|\psi(u_{n_1}^{(1)},\dots,u_{n_k}^{(k)})|^p \le c_{r',p}^{-kp}\Big(\int_{\mathbb
K^N}\dots\int_{\mathbb K^N}|\psi(z^{(1)},\dots,z^{(k)})|^pd\mu^N_{r'}(z^{(1)})\dots d\mu^N_{r'}(z^{(k)})\Big),
\end{eqnarray*}
for all sequences   $(u_{n_j}^{(j)})\subset \ell_r^N$, with $w_p(u_{n_j}^{(j)})=1$, $j=1,\dots, k$.

Let $\phi$ be an $m$-linear form, and $(u_{n_j}^{(j)})\subset \ell_r^N$, with $w_p(u_{n_j}^{(j)})=1$,
$j=1,\dots, m$. Then
\begin{align*}
\sum_{n_1,\dots,n_m} & |\phi(u_{n_1}^{(1)},\dots,u_{n_m}^{(m)})|^p  =
\sum_{n_1}\sum_{n_2,\dots,n_m}|\phi(u_{n_1}^{(1)},\dots,u_{n_m}^{(m)})|^p \\
 & \le  c_{r',p}^{-(m-1)p}\sum_{n_1}\Big(\int_{\mathbb K^N}\dots\int_{\mathbb
K^N}|\phi({u_{n_1}^{(1)}},z^{(2)},\dots,z^{(m)})|^pd\mu^N_{r'}(z^{(2)})\dots d\mu^N_{r'}(z^{(m)})\Big) \\
&=  c_{r',p}^{-(m-1)p}\Big(\int_{\mathbb K^N}\dots\int_{\mathbb
K^N}\sum_{n_1}|\phi({u_{n_1}^{(1)}},z^{(2)},\dots,z^{(m)})|^pd\mu^N_{r'}(z^{(2)})\dots
d\mu^N_{r'}(z^{(m)})\Big) \\
&\le  c_{r',p}^{-(m-1)p}\int_{\mathbb K^N}\dots\int_{\mathbb K^N} \Big(c_{r',p}^{-p}\int_{\mathbb
K^N}|\phi(z^{(1)},z^{(2)},\dots,z^{(m)})|^pd\mu^N_{r'}(z^{(1)})\Big) d\mu^N_{r'}(z^{(2)})\dots d\mu^N_{r'}(z^{(m)})  \\
&=c_{r',p}^{-mp}\Big(\int_{\mathbb K^N}\dots\int_{\mathbb K^N} |\phi(z^{(1)},\dots,z^{(m)})|^p
d\mu_{r'}^N(z^{(1)})\dots d\mu^N_{r'}(z^{(m)})\Big) .
\end{align*}
Therefore,
$$
c_{r',p}^m\pi_p(\phi)\le \Big(\int_{\mathbb K^N}\dots\int_{\mathbb
K^N}|\phi(z^{(1)},\dots,z^{(m)})|^pd\mu^N_{r'}(z^{(1)})\dots d\mu^N_{r'}(z^{(m)})\Big)^{1/p}. \qedhere
$$
\end{proof}

\medskip

Let us continue  our way to the proof of Theorem~\ref{formulita}.

\begin{lemma}
Let $T$ be an $m$-linear mapping in $\mathcal L(^mE;\ell_q^M)$, $0<p<q<2$ or $q=2$. Then
$$
c_{q,p}\pi_p(T)\le \Big(\int_{\mathbb K^M}\pi_p(z\circ T)^pd\mu^M_{q}(z)\Big)^{1/p}.
$$
In particular, if $T$ is linear,
$$
c_{q,p}\pi_p(T)\le \Big(\int_{\mathbb K^M}\|T'(z)\|_{E'}^p \, d\mu^M_{q}(z)\Big)^{1/p}.
$$
\end{lemma}
\begin{proof}
For $(u_{k_j}^j)\subset \ell_2^N$ with $w_p((u_{k_j}^j))=1$, $j=1, \ldots, m,$ we have
\begin{eqnarray*}
\sum_{k_1,\dots,k_m} \|T(u_{k_1}^1,\dots,u_{k_m}^m)\|_{\ell_q^M}^p & = &  \sum_{k_1,\dots,k_m} \Big(\sum_{j=1}^M |e_j'\circ T(u_{k_1}^1,\dots,u_{k_m}^m)|^q\Big)^{p/q} \\
 & = & \sum_{k_1,\dots,k_m} c_{q,p}^{-p}\Big(\int_{\mathbb K^M}\Big|\sum_{j=1}^M e_j'\circ T(u_{k_1}^1,\dots,u_{k_m}^m)
z_j\Big|^pd\mu^M_{q}(z)\Big)  \\
 & = & \sum_{k_1,\dots,k_m} c_{q,p}^{-p}\Big(\int_{\mathbb K^M}|z\circ
T(u_{k_1}^1,\dots,u_{k_m}^m)|^pd\mu^M_{q}(z)\Big)
 \\
  & \le & c_{q,p}^{-p}\Big(\int_{\mathbb K^M}\pi_p(z\circ T)^pd\mu^M_{q}(z)\Big) .
\end{eqnarray*}
Therefore,
$$
c_{q,p}\pi_p(T)\le \Big(\int_{\mathbb K^M}\pi_p(z\circ T)^pd\mu^M_{q}(z)\Big)^{1/p}.
$$

For $m=1$, $z\circ T$ is a linear form, and then we have $\pi_p(z\circ T)=\|z\circ T\|_{E'}=\|T'(z)\|_{E'}$.
\end{proof}

By a {\it Banach sequence space} we mean a Banach space $X\subset \mathbb K^{\mathbb N}$ of sequences in $\mathbb K$ such that $\ell_1\subset X\subset \ell_\infty$ with norm one inclusions satisfying that if 
$x\in\mathbb K^{\mathbb N}$ and $y\in X$ are such that $|x_n|\le |y_n|$ for every $n\in\mathbb N$, then $x$ belongs to $X$ and $\|x\|_X\le\|y\|_X$.
We will now show that if we consider multilinear mappings whose range are certain Banach sequence spaces, then
the norm of the multilinear mapping defined by the 
integral formula is equivalent to the multiple summing norm.

We will need the following remark, which may be seen as a Khintchine/Kahane type  multilinear inequality for the stable
measures.
\begin{remark}\label{khintchine estable}
If $T$ is an $m$-linear form on $\mathbb K^N$ and $q\le p<s<2$, or $q\le p$ and $s=2$, then
\begin{multline*}
 c_{s,p}^{-m}\Big(\int_{\mathbb K^N}\dots\int_{\mathbb
K^N}|T(z^{(1)},\dots,z^{(m)})|^pd\mu^N_{s}(z^{(1)})\dots d\mu^N_{s}(z^{(m)})\Big)^{1/p} \\
 \le c_{s,q}^{-m}\Big(\int_{\mathbb K^N}\dots\int_{\mathbb
K^N}|T(z^{(1)},\dots,z^{(m)})|^qd\mu^N_{s}(z^{(1)})\dots d\mu^N_{s}(z^{(m)})\Big)^{1/q}.
\end{multline*}
For $m=1$ it follows from property \eqref{estable} of L\'evy stable measures, and then we just apply induction on $m$.
\end{remark}

Recall that a Banach sequence space $X$ is called {\it $q$-concave}, $q\ge1$, if there
exists $C>0$ such that for any $x_1,\dots,x_n\in X$ we have
$$
\Big(\sum_{k=1}^n\|x_k\|_X^q\Big)^\frac1{q}\le C\Big\|\Big(\sum_{k=1}^n|x_k|^q\Big)^{1/q}\Big\|_X.
$$
\begin{lemma}\label{lema1}
 Let $X$ be a $q$-concave Banach sequence space with constant $C$ and let
$T\in\mathcal
L(^mE;X)$ be an $m$-linear operator. Denote by $T_j$ the $j$-coordinate of
$T$ ($T_j$ is a scalar $m$-linear form). Then $\pi_q(T)\le C \|(\pi_q(T_j))_j\|_X$.
\end{lemma}
\begin{proof} Just note that for finite sequences $\big(u_{n_k}^{(k)}\big)_{n_k}\subset X$, with
$w_q\Big(\big(u_{n_k}^{(k)}\big)_{n_k}\Big)=1$ we have
 $$
\Big(\sum_{n_1,\dots,n_m}\|T(u_{n_1}^{(1)},\dots,u_{n_m}^{(m)})\|_X^q\Big)^{1/q} \le C
\Big\|\Big(\big(\sum_{n_1,\dots,n_m}|T_j(u_{n_1}^{(1)},\dots,u_{n_m}^{(m)})|^q\big)^{1/q}\Big)_j\Big\|_X \le
C\big\|(\pi_q(T_j))_j\big\|_X. $$
\end{proof}

\begin{lemma}\label{lema2}
Let $X$ be a Banach sequence space  and let $T\in\mathcal
L(^m\ell_r^N;X)$ be an $m$-linear operator, $r\ge2$.
Then if either $p,q<r'<2$, or $r=2$, then
$$\|(\pi_p(T_j))_j\|_X\le c_{r',1}^{-m} \int_{\mathbb K^N}\dots\int_{\mathbb
K^N}\|T(z^{(1)},\dots,z^{(m)})\|_Xd\mu^N_{r'}(z^{(1)})\dots d\mu^N_{r'}(z^{(m)})\le
(c_{r',q}/c_{r',1})^{m} \pi_q(T).$$
\end{lemma}
\begin{proof}
 By Theorem \ref{formulita escalar}, Remark \ref{khintchine
estable}
and  Lemma \ref{formulita desig facil} we have
\begin{eqnarray*}
\|(\pi_p(T_j))_j\|_X &= & c_{r',p}^{-m}\Big\|\left(\Big(\int_{\mathbb K^N}\dots\int_{\mathbb
K^N}|T_j(z^{(1)},\dots,z^{(m)})|^pd\mu^N_{r'}(z^{(1)})\dots d\mu^N_{r'}(z^{(m)})\Big)^{1/p}\right)_j\Big\|_X
\\
&\le & c_{r',1}^{-m}\Big\|\Big(\int_{\mathbb K^N}\dots\int_{\mathbb
K^N}|T_j(z^{(1)},\dots,z^{(m)})|d\mu^N_{r'}(z^{(1)})\dots d\mu^N_{r'}(z^{(m)})\Big)\Big\|_X \\
&\le & c_{r',1}^{-m} \int_{\mathbb K^N}\dots\int_{\mathbb
K^N}\|T(z^{(1)},\dots,z^{(m)})\|_Xd\mu^N_{r'}(z^{(1)})\dots d\mu^N_{r'}(z^{(m)})\\
&\le & c_{r',1}^{-m} \Big(\int_{\mathbb K^N}\dots\int_{\mathbb
K^N}\|T(z^{(1)},\dots,z^{(m)})\|_X^qd\mu^N_{r'}(z^{(1)})\dots d\mu^N_{r'}(z^{(m)})\Big)^{1/q}\\
&\le & (c_{r',q}/c_{r',1})^{m} \pi_q(T).
\end{eqnarray*}
\end{proof}

As a consequence of Lemma \ref{formulita desig facil} we obtain one  inequality in the following result. For the other inequality, note that if $X$ is $q$-concave, then it is also $s$-concave for any $s\ge q$ and apply the previous two lemmas.

\begin{corollary}\label{formulita q-concave}
 Let $X$ be a $q$-concave Banach sequence space and let $T\in\mathcal
L(^m\ell_r^N;X)$. Then for $r=2$ and $q\le s$, or $q\le s<r'<2$, we have
\begin{eqnarray*} \pi_s(T)   \asymp
 \Big(\int_{\mathbb K^N}\dots\int_{\mathbb
K^N}\|T(z^{(1)},\dots,z^{(m)})\|_X^qd\mu^N_{r'}(z^{(1)})\dots d\mu^N_{r'}(z^{(m)})\Big)^{1/q}  \\
\end{eqnarray*}
\end{corollary}

Standard localization techniques and the previous corollary readily show the coincidence of multiple
$s-$summing and multiple $q-$summing operators from  $\mathcal L_r^g$-spaces to $q$-concave Banach sequence
spaces.

\begin{corollary}
 Let $X$ be a $q$-concave Banach sequence space, and let $E$ be an $\mathcal L_r^g$-space.
Then
$$
\Pi_s(^mE;X)=\Pi_q(^mE;X),
$$
for $q\le s<r'<2$, or $q\le s$ and $r=2$.	
\end{corollary}

 Proceeding as above we may prove the following.
 \begin{corollary}\label{pi_q en L_q}
Let $X$ be an $\mathcal L_{q,1}^g$-space and let $T\in\mathcal
L(^m\ell_r^N;X)$ be an $m$-linear operator. If $q<r'<2$ or,
$r=2$, then we have	
$$\pi_q(T) = c_{r',q}^{-m}\Big(\int_{\mathbb K^N}\dots\int_{\mathbb
K^N}\|T(z^{(1)},\dots,z^{(m)})\|_{X}^qd\mu^N_{r'}(z^{(1)})\dots d\mu^N_{r'}(z^{(m)})\Big)^{1/q}.$$
\end{corollary}

We have almost finished the proofs of the main results.
\begin{proof}[Proof of Theorem \ref{formulita}]
It is clearly enough to show the result for operators with range in $\ell_q^M$ for some $M\in\mathbb N$.

One inequality is Lemma \ref{formulita desig facil}.
For the other inequality, if either $r=q=2$ or; $r=2$ and $p<q<2$ or;  $p<r'<2$ and $p<q\le2$, a
combination
of the previous results gives:
\begin{eqnarray*}
 \pi_p(T) & \le & c_{q,p}^{-1}\Big(\int_{\mathbb K^N}\pi_p(z\circ T)^pd\mu^N_{q}(z)\Big)^{1/p} \\
 & \le & c_{r',p}^{-m}c_{q,p}^{-1}\Big(\int_{\mathbb K^N}\Big(\int_{\mathbb K^N}\dots\int_{\mathbb K^N}|z\circ
T(z^{(1)},\dots,z^{(m)})|^pd\mu^N_{r'}(z^{(1)})\dots d\mu^N_{r'}(z^{(m)})\Big)d\mu^N_{q}(z)\Big)^{1/p} \\
& \le & c_{r',p}^{-m}\Big(\int_{\mathbb K^N}\dots\int_{\mathbb K^N}\Big(c_{q,p}^{-p}\int_{\mathbb K^N}|z\circ
T(z^{(1)},\dots,z^{(m)})|^pd\mu^N_{q}(z)\Big)d\mu^N_{r'}(z^{(1)})\dots d\mu^N_{r'}(z^{(m)})\Big)^{1/p} \\
& \le & c_{r',p}^{-m}\Big(\int_{\mathbb K^N}\dots\int_{\mathbb
K^N}\pi_p\big(T(z^{(1)},\dots,z^{(m)});\ell_{q'}^M\to\mathbb K\big)^pd\mu^N_{r'}(z^{(1)})\dots
d\mu^N_{r'}(z^{(m)})\Big)^{1/p} \\
& = & c_{r',p}^{-m}\Big(\int_{\mathbb K^N}\dots\int_{\mathbb
K^N}\|T(z^{(1)},\dots,z^{(m)})\|_{\ell_q^M}^pd\mu^N_{r'}(z^{(1)})\dots d\mu^N_{r'}(z^{(m)})\Big)^{1/p}, \\
\end{eqnarray*}
where by $\pi_p\big(T(z^{(1)},\dots,z^{(m)});\ell_{q'}^M\to\mathbb K\big)$ we denote the absolutely
$p$-summing norm of the vector $T(z^{(1)},\dots,z^{(m)})$ thought of as a linear functional on $\ell_{q'}^M$,
whose norm is just the $\ell_{q}^M$-norm of the vector.

The cases where $p=q$ follow from Corollary \ref{pi_q en L_q}.
\end{proof}
\begin{proof}[Proof of Proposition \ref{normas equivalentes}]
($i$)
For $r=2$, the equivalence of norms is a consequence of Theorem \ref{formulita} when $p<q\le 2$ or $p=q$ and
of
Corollary \ref{formulita q-concave} for $q\le p$.

For $r>2$, the equivalence of norms is a consequence of Theorem \ref{formulita} when $p<r'$ and $p<q\le 2$ and
of
Corollary \ref{formulita q-concave} $q\le p<r'$.

($ii$) This statement follows from Lemma \ref{formulita desig facil}.
\end{proof}

\begin{proof}[Proof of Proposition \ref{prop inclusion}] The first assertion follows from Corollary~\ref{formulita q-concave} and localization. For the second assertion just combine  Lemma \ref{lema1} with Lemma \ref{lema2}.
\end{proof}

\end{document}